\documentclass[reqno]{amsart}
 \usepackage[foot]{amsaddr}
\usepackage{hyperref}
\usepackage[english]{babel}
\usepackage[utf8]{inputenc}
\usepackage{cite}
\usepackage{graphicx}

\numberwithin{equation}{section}

\theoremstyle{definition}
\newtheorem{Def}{Definition}[section]
\newtheorem{Rem}{Remark}[section]
\newtheorem{Exax}{Example}[section]	
\newenvironment{Exa}
  {\pushQED{\qed}\Exax}
  {\popQED\endExax}
  
\theoremstyle{plain}
\newtheorem{The}{Theorem}[section]
\newtheorem{Pro}{Proposition}[section]

\newcommand{\bb}[1]{\mathbb{#1}}
\newcommand{\norm}[1]{\left\|{#1}\right\|}
\newcommand{\abs}[1]{\left|{#1}\right|}
\newcommand{\scalarp}[1]{\left\langle{#1}\right\rangle}

\newcommand{\Remref}[1]{Remark \ref{#1}}
\newcommand{\Figref}[1]{Fig. \ref{#1}}
\newcommand{\Exaref}[1]{Example \ref{#1}}
\newcommand{\Appref}[1]{Appendix \ref{#1}}

\newcommand{\Secref}[1]{Section \ref{#1}}
\newcommand{\Defref}[1]{Definition \ref{#1}}
\newcommand{\Theref}[1]{Theorem \ref{#1}}
\newcommand{\Proref}[1]{Proposition \ref{#1}}

\newcommand{\Defenuref}[2]{Definition \ref{#1}\eqref{#2}}
\newcommand{\Theenuref}[2]{Theorem \ref{#1}\eqref{#2}}

\begin{document}

\title{Nonlinear approximation with nonstationary Gabor frames}

\author{Emil Solsb{\ae}k Ottosen}
\author{Morten Nielsen}

\address{Department of Mathematical Sciences, Aalborg University, Skjernvej 4, 9220 Aalborg {\O}, Denmark}
\email{emilo@math.aau.dk}
\email{mnielsen@math.aau.dk}


\begin{abstract}
We consider sparseness properties of adaptive time-frequency representations obtained using nonstationary Gabor frames (NSGFs). NSGFs generalize classical Gabor frames by allowing for adaptivity in either time or frequency. It is known that the concept of painless nonorthogonal expansions generalizes to the nonstationary case, providing perfect reconstruction and an FFT based implementation for compactly supported window functions sampled at a certain density. It is also known that for some signal classes, NSGFs with flexible time resolution tend to provide sparser expansions than can be obtained with classical Gabor frames. In this article we show, for the continuous case, that sparseness of a nonstationary Gabor expansion is equivalent to smoothness in an associated decomposition space. In this way we characterize signals with sparse expansions relative to NSGFs with flexible time resolution. Based on this characterization we prove an upper bound on the approximation error occurring when thresholding the coefficients of the corresponding frame expansions. We complement the theoretical results with numerical experiments, estimating the rate of approximation obtained from thresholding the coefficients of both stationary and nonstationary Gabor expansions.
\end{abstract}

\maketitle

\section{Introduction}\label{PAPERC:Sec:1}
The field of Gabor theory \cite{Gabor1946,Daubechies86,Young80} is concerned with representing signals as atomic decompositions using time-frequency localized atoms. The atoms are constructed as time-frequency shifts of a fixed window function, according to some lattice parameters, such that the resulting system constitutes a frame and, therefore, guarantees stable expansions \cite{Christensen2016,Ron97,Mallat2009}. Such frames are known under the name of Weyl-Heisenberg frames or Gabor frames and have been proven useful in a variety of applications \cite{Grochenig2001,Dorfler2001,Pfander2013}. The structure of Gabor frames implies a time-frequency resolution which depends only on the lattice parameters and the window function. In particular, the resolution is independent of the signal under consideration, which makes the corresponding implementation fast and easy to handle. The usage of a predetermined time-frequency resolution naturally raises the question of whether an improvement can be obtained by taking the signal class into consideration? This question has lead to many interesting approaches for constructing adaptive time-frequency representations \cite{Dorfler2011Quilt,Jaillet2007,Wolfe2001,Zibulski1997}. Unfortunately, for representations with resolution varying in \emph{both} time and frequency there seems to be a trade-off between perfect reconstruction and fast implementation \cite{Liuni2013}. In this article, we therefore consider time-frequency representations with resolution varying in \emph{either} time or frequency. The idea is to generalise the theory of painless nonorthogonal expansions \cite{Daubechies86} to the situation where multiple window functions are used along either the time- or the frequency axis. The resulting systems, which allow for perfect reconstruction and an FFT based implementation, are called painless generalised shift-invariant systems \cite{Ron2005,Hernandez2002} or painless nonstationary Gabor frames (painless NSGFs) \cite{Holighaus2014,Balazs11}. As already noted in \cite{Balazs11}, painless NSGFs tend to produce sparser representations than classical Gabor frames for certain classes of music signals. Sparseness of a time-frequency representation is desirable for several reasons, mainly because it may reduce the computational cost for manipulating and storing the coefficients \cite{Foucart2013,Devore1993}. Additionally, many signal classes are characterized by some kind of sparseness in time or frequency and the corresponding signals are, therefore, best described by a sparse time-frequency representation. For such signals, the task of feature identification also benefits from a sparse representation as the particular characteristics of the signal becomes easier to identify. 

In this article we consider sparseness properties of painless NSGFs with resolution varying in time. Whereas modulation spaces \cite{Feichtinger83,Grochenig2001,Feichtinger2006} have turned out to be the proper function spaces for analyzing sparseness properties of classical Gabor frames \cite{Grochenig2000}, we need a more general framework for the nonstationary case. A painless NSGF with flexible time resolution corresponds to a sampling grid which is irregular over time but regular over frequency for each fixed time point. We therefore search for a smoothness space which is compatible with a (more or less) arbitrary partition of the time domain. Such a flexibility can be provided by decomposition spaces, as introduced by Feichtinger and Gr\"{o}bner in \cite{Feichtinger1985,Feichtinger1987}. Decomposition spaces may be viewed as a generalization of the classical Wiener amalgam spaces \cite{Feichtinger83Wiener,Heil2003} but with no assumption of an upper bound on the measure of the members of the partition. Another way of stating this is that decomposition spaces are constructed using bounded \emph{admissible} partitions of unity \cite{Feichtinger1985} instead of bounded \emph{uniform} partitions of unity \cite{Feichtinger83Wiener}. The partitions we consider are obtained by applying a set of invertible affine transformations $\{A_k(\cdot)+c_k\}_{k\in \bb{N}}$ on a fixed set $Q\subset \bb{R}^d$ \cite{Borup07}.

We use decomposition spaces to characterize signals with sparse expansions relative to painless NSGFs with flexible time resolution. We measure sparseness of an expansion by a mixed norm on the coefficients and show that the sparseness property implies an upper bound on the approximation error obtained by thresholding the expansion. Using the terminology from nonlinear approximation, such an upper bound is also known as a Jackson inequality \cite{Devore1993,Butzer72}. A similar characterization for classical Gabor frames using modulation spaces was proven by Gr\"ochenig and Samarah in \cite{Grochenig2000}. For the nonstationary case, we provided a characterization in \cite{Ottosen2017Article1} for painless NSGFs with flexible frequency resolution using decomposition spaces. A different approach to this problem is considered by Voigtlaender in \cite{Voigtlaender2016}, where the painless assumption is replaced with a more general analysis of the sampling parameter. The decomposition spaces considered in both \cite{Ottosen2017Article1} and \cite{Voigtlaender2016} are based on partitions of the frequency domain, which is not a natural choice for NSGFs with flexible time resolution. In this article we consider decompositions of the time domain, which allow for compactly supported window functions sampled at a low density (compared to the general theory formulated in \cite{Voigtlaender2016}). It is worth noting that there is a significant mathematical difference between decomposition spaces in time and in frequency.

The structure of this this article is as follows. In \Secref{PAPERC:Sec:2} we formally introduce decomposition spaces in time and prove several important properties of these spaces. Then, based on the ideas in \cite{Ottosen2017Article1}, we show in \Secref{PAPERC:Sec:3} how to construct a suitable decomposition space for a given painless NSGF with flexible time resolution. In \Secref{PAPERC:Sec:4} we prove that the suitable decomposition space characterizes signals with sparse frame expansions and we provide an upper bound on the approximation rate occurring when thresholding the frame coefficients. Finally, in \Secref{PAPERC:Sec:5} we present the numerical results and in \Secref{PAPERC:Sec:6} we give the conclusions. 

Let us now briefly go through our notation. By $\hat{f}(\xi):=\int_{\bb{R}^d}f(x)e^{-2\pi i x\cdot\xi}dx$ we denote the Fourier transform with the usual extension to $L^2(\bb{R}^d)$. With $F\asymp G$ we mean that there exist two constants $0<C_1,C_2<\infty$ such that $C_1F\leq G\leq C_2F$. For two normed vector spaces $X$ and $Y$, $X\hookrightarrow Y$ means that $X\subset Y$ and $\norm{f}_Y\leq C\norm{f}_X$ for some constant $C$ and all $f\in X$. We say that a non-empty open set $\Omega'\subset \bb{R}^d$ is \emph{compactly contained} in an open set $\Omega\subset \bb{R}^d$ if $\overline{\Omega'}\subset \Omega$ and $\overline{\Omega'}$ is compact. We call $\{x_i\}_{i\in \mathcal{I}}\subset \bb{R}^d$ a $\delta-$separated set if $\inf_{j,k\in \mathcal{I}, j\neq k}\|x_j-x_k\|_2=\delta>0$. Finally, by $I_d$ we denote the identity operator on $\bb{R}^d$ and by $\chi_Q$ we denote the indicator function for a set $Q\subset \bb{R}^d$. 

\section{Decomposition spaces}\label{PAPERC:Sec:2}
In this section we define decomposition spaces \cite{Feichtinger1985} based on structured coverings \cite{Borup07}. For an invertible matrix $A\in GL(\bb{R}^d)$, and a constant $c\in \bb{R}^d$, we define the affine transformation $Tx=Ax+c$ with $x\in \bb{R}^d$. Given a family $\mathcal{T}=\{A_k(\cdot)+c_k\}_{k\in \bb{N}}$ of invertible affine transformations on $\bb{R}^d$, and a subset $Q\subset \bb{R}^d$, we let $\{Q_T\}_{T\in \mathcal{T}}:=\{T(Q)\}_{T\in \mathcal{T}}$ and
\begin{equation}\label{PAPERC:eq:neighbouringindices}
\widetilde{T}:=\left\{T'\in \mathcal{T}~\big|~ Q_{T'}\cap Q_T \neq \emptyset\right\},\quad T\in \mathcal{T}.
\end{equation}
We say that $\mathcal{Q}:=\{Q_T\}_{T\in \mathcal{T}}$ is an \emph{admissible covering} of $\bb{R}^d$ if $\bigcup_{T\in \mathcal{T}}Q_T=\bb{R}^d$ and there exists $n_0\in \bb{N}$ such that $|\widetilde{T}|\leq n_0$ for all $T\in \mathcal{T}$.
\begin{Def}[$\mathcal{Q}-$moderate weight]\label{PAPERC:Def:moderateweight}
Let $\mathcal{Q}:=\{Q_T\}_{T\in \mathcal{T}}$ be an admissible covering. A function $u:\bb{R}^d\rightarrow (0,\infty)$ is called $\mathcal{Q}-$moderate if there exists $C>0$ such that $u(x)\leq Cu(y)$ for all $x,y\in Q_T$ and all $T\in \mathcal{T}$. A $\mathcal{Q}-$moderate weight (derived from $u$) is a sequence $\{\omega_T\}_{T\in \mathcal{T}}:=\{u(x_T)\}_{T\in \mathcal{T}}$ with $x_T\in Q_T$ for all $T\in \mathcal{T}$. 
\end{Def}
For the rest of this article, we shall use the explicit choice $u(x):=1+\|x\|_2$ for the function $u$ in \Defref{PAPERC:Def:moderateweight}. Let us now define structured coverings \cite{Borup07} of the time domain.
\begin{Def}[Structured covering]\label{PAPERC:Def:SAC}
Given a family $\mathcal{T}=\{A_k(\cdot)+c_k\}_{k\in \bb{N}}$ of invertible affine transformations on $\bb{R}^d$, suppose there exist two bounded open sets $P\subset Q\subset \bb{R}^d$, with $P$ compactly contained in $Q$, such that
\begin{enumerate}
\item $\left\{P_T\right\}_{T\in \mathcal{T}}$ and $\left\{Q_T\right\}_{T\in \mathcal{T}}$ are admissible coverings.\label{PAPERC:Def:SAC1}

\item There exists a $\delta-$separated set $\{x_T\}_{T\in \mathcal{T}}\subset \bb{R}^d$, with $x_T\in Q_T$ for all $T\in \mathcal{T}$, such that $\{\omega_T\}_{T\in \mathcal{T}}:=\{1+\|x_T\|_2\}_{T\in \mathcal{T}}$ is a $\mathcal{Q}-$moderate weight.\label{PAPERC:Def:SAC2}
\end{enumerate}
Then we call $\mathcal{Q}=\{Q_T\}_{T\in \mathcal{T}}$ a \emph{structured  covering}.
\end{Def}
For a structured covering we have the associated concept of a \emph{bounded admissible partition of unity} (BAPU) \cite{Feichtinger1985}.
\begin{Def}[BAPU]\label{PAPERC:Def:Bapu}
Let $\mathcal{Q}=\{Q_T\}_{T\in \mathcal{T}}$ be a structured covering of $\bb{R}^d$. A BAPU subordinate to $\mathcal{Q}$ is a family of non-negative functions $\{\psi_T\}_{T\in \mathcal{T}}\subset C^\infty_c(\bb{R}^d)$ satisfying
\begin{enumerate}
\item $\text{supp}(\psi_T)\subset Q_T,\quad \forall T\in \mathcal{T}$.\label{PAPERC:Def:Bapu1}

\item $\displaystyle \sum_{T \in \mathcal{T}}\psi_T(x)=1, \quad \forall x\in \bb{R}^d$.\label{PAPERC:Def:Bapu2}
\end{enumerate}
\end{Def}
We note that the assumptions in \Defref{PAPERC:Def:Bapu} implies that the members of the BAPU are uniformly bounded, i.e., $\sup_{T\in \mathcal{T}}\norm{\psi_T}_{L^\infty}\leq 1$. 

Given a structured covering $\mathcal{Q}=\{Q_T\}_{T\in \mathcal{T}}$, we can always construct a subordinate BAPU. Choose a non-negative function $\Phi\in C^\infty_c(\bb{R}^d)$, with $\Phi(x)=1$ for all $x\in P$ and $\text{supp}(\Phi)\subset Q$, and define
\begin{equation*}
\psi_T(x):=\frac{\Phi(T^{-1}x)}{\sum_{T'\in \mathcal{T}}\Phi(T'^{-1}x)},\quad x\in \bb{R}^d,
\end{equation*}
for all $T\in \mathcal{T}$. With this construction, it is clear that \Defenuref{PAPERC:Def:Bapu}{PAPERC:Def:Bapu1} is satisfied. Further, since $\{P_T\}_{T\in \mathcal{T}}$ is an admissible covering, then $1\leq \sum_{T'\in \mathcal{T}}\Phi(T'^{-1}x)\leq n_0$ for all $x\in \bb{R}^d$ which shows that \Defenuref{PAPERC:Def:Bapu}{PAPERC:Def:Bapu2} holds. 
\begin{Rem}
We note that the assumption in \Defenuref{PAPERC:Def:SAC}{PAPERC:Def:SAC2} is not necessary for constructing a subordinate BAPU, however, the assumption is needed for proving \Theref{PAPERC:The:decomcovthe}.
\end{Rem}
Let $\mathcal{Q}=\{Q_T\}_{T\in \mathcal{T}}$ be a structured covering with $\mathcal{Q}-$moderate weight $\{\omega_T\}_{T\in \mathcal{T}}=\{1+\|x_T\|_2\}_{T\in \mathcal{T}}$ and BAPU $\{\psi_T\}_{T\in \mathcal{T}}$. For $s\in \bb{R}$ and $1\leq q\leq\infty$, we define the associated weighted sequence space 
\begin{equation*}
\ell^q_{\omega^s}(\mathcal{T}):=\left\{\{a_T\}_{T\in \mathcal{T}}\subset \bb{C}~\Big|~\norm{\{a_T\}_{T\in \mathcal{T}}}_{\ell^q_{\omega^s}}:=\norm{\{\omega_T^sa_T\}_{T\in \mathcal{T}}}_{\ell^q}<\infty\right\}.
\end{equation*}
Given $\{a_T\}_{T\in \mathcal{T}}\in \ell^q_{\omega^s}(\mathcal{T})$, we define $\{a_T^+\}_{T\in \mathcal{T}}$ by $a_T^+:=\sum_{T'\in \widetilde{T}}a_{T'}$. Since $\{\omega_T\}_{T\in \mathcal{T}}$ is $\mathcal{Q}-$moderate, $\{a_T\}_{T\in \mathcal{T}}\rightarrow \{a_T^+\}_{T\in \mathcal{T}}$ defines a bounded operator on $\ell^q_{\omega^s}(\mathcal{T})$ according to \cite[Remark 2.13 and Lemma 3.2]{Feichtinger1985}. Denoting its operator norm by $C_+$, we have 
\begin{equation}\label{PAPERC:eq:Operatornorm}
\norm{\left\{a_T^+\right\}_{T\in \mathcal{T}}}_{\ell^q_{\omega^s}}\leq C_+ \norm{\left\{a_T\right\}_{T\in \mathcal{T}}}_{\ell^q_{\omega^s}},\quad \forall \left\{a_T\right\}_{T\in \mathcal{T}}\in \ell^q_{\omega^s}(\mathcal{T}).
\end{equation}
We now define decomposition spaces as first introduced in \cite{Feichtinger1985}.
\begin{Def}[Decomposition space]\label{PAPERC:Def:DS}
Let $\mathcal{Q}=\{Q_T\}_{T\in \mathcal{T}}$ be a structured covering with $\mathcal{Q}-$moderate weight $\{\omega_T\}_{T\in \mathcal{T}}=\{1+\|x_T\|_2\}_{T\in \mathcal{T}}$ and BAPU $\{\psi_T\}_{T\in \mathcal{T}}$. For $s\in \bb{R}$ and $1\leq p,q\leq\infty$, we define the \emph{decomposition space} $D(\mathcal{Q},L^p,\ell^q_{\omega^s})$ as the set of distributions $f\in \mathcal{S}'(\bb{R}^d)$ satisfying
\begin{equation*}
\norm{f}_{D(\mathcal{Q},L^p,\ell^q_{\omega^s})}:=\norm{\left\{\norm{\psi_T f}_{L^p}\right\}_{T\in \mathcal{T}}}_{\ell^q_{\omega^s}}<\infty.
\end{equation*}	
\end{Def}
\begin{Rem}\label{PAPERC:Rem:equivBAPU}
According to \cite[Theorem 3.7]{Feichtinger1985}, $D(\mathcal{Q},L^p,\ell^q_{\omega^s})$ is independent of the particular choice of BAPU and different choices yield equivalent norms. Actually the results in \cite{Feichtinger1985} show that $D(\mathcal{Q},L^p,\ell^q_{\omega^s})$ is invariant under certain geometric modifications of $\mathcal{Q}$, but we will not go into detail here.
\end{Rem}
\begin{Rem}
In contrast to the approach taken in \cite{Ottosen2017Article1} (where the decomposition is performed on the frequency side), we do not allow $p,q<1$ in \Defref{PAPERC:Def:DS} since a simple consideration shows that  the resulting decomposition spaces would not be complete in this case.
\end{Rem}
We now consider some familiar examples of decomposition spaces. By standard arguments it is easy to verify that $D(\mathcal{Q},L^2,\ell^2)=L^2(\bb{R}^d)$ with equivalent norms for any structured covering $\mathcal{Q}$. The next example shows how to construct Wiener amalgam spaces.
\begin{Exa}\label{PAPERC:Exa:wieneramalgam}
Let $Q\subset \bb{R}^d$ be an open cube with center $0$ and side-length $r>1$. Define $\mathcal{T}:=\{T_k\}_{k\in \bb{Z}^d}$, with $T_kx:=x-k$ for all $k\in \bb{Z}^d$, and let $\{\omega_{T_k}\}_{T_k\in \mathcal{T}}=\{1+\|k\|_2\}_{T_k\in \mathcal{T}}$. With $\mathcal{Q}:=\{Q_{T_k}\}_{T_k\in \mathcal{T}}$, then $D(\mathcal{Q},L^p,\ell^q_{\omega^s})$ corresponds to the Wiener amalgam space $W(L^p,\ell^q_{\omega^s})$ for $s\in \bb{R}$ and $1\leq p,q\leq\infty$, see \cite{Feichtinger83Wiener} for further details.
\end{Exa}
Let us now prove the following important properties of decomposition spaces.
\begin{The}\label{PAPERC:The:decomcovthe}
Let $\mathcal{Q}=\{Q_T\}_{T\in\mathcal{T}}$ be a structured covering with $\mathcal{Q}-$moderate weight $\{\omega_T\}_{T\in \mathcal{T}}=\{1+\|x_T\|_2\}_{T\in \mathcal{T}}$ and subordinate BAPU $\{\psi_T\}_{T\in\mathcal{T}}$. For $s\in \bb{R}$ and $1\leq p,q\leq \infty$,
\begin{enumerate}
\item $\mathcal{S}(\bb{R}^d)\hookrightarrow D(\mathcal{Q},L^p,\ell^q_{\omega^s}) \hookrightarrow \mathcal{S}'(\bb{R}^d)$.\label{PAPERC:The:decomcovthe1}

\item $D(\mathcal{Q},L^p,\ell^q_{\omega^s})$ is a Banach space.\label{PAPERC:The:decomcovthe2}

\item If $p,q<\infty$, then $\mathcal{S}(\bb{R}^d)$ is dense in $D(\mathcal{Q},L^p,\ell^q_{\omega^s})$.\label{PAPERC:The:decomcovthe3}

\item If $p,q<\infty$, then the dual space of $D(\mathcal{Q},L^p,\ell^q_{\omega^s})$ can be identified with $D(\mathcal{Q},L^{p'},\ell^{q'}_{\omega^{-s}})$ with $1/p+1/p'=1$ and $1/q+1/q'=1$. \label{PAPERC:The:decomcovthe4}
\end{enumerate}
\end{The}
The proof of \Theref{PAPERC:The:decomcovthe} can be found in \Appref{PAPERC:App:1}. In the next section we construct decomposition spaces, which are compatible with the structure of painless NSGFs with flexible time resolution.

\section{Nonstationary Gabor frames}\label{PAPERC:Sec:3}
In this section, we construct NSGFs with flexible time resolution using the notation of \cite{Balazs11}. Given a set of window functions $\{g_n\}_{n\in \bb{Z}^d}\subset L^2(\bb{R}^d)$, with corresponding frequency sampling steps $b_n>0$, then for $m,n\in \bb{Z}^d$ we define atoms of the form
\begin{equation*}
g_{m,n}(x):=g_n(x)e^{2\pi imb_n\cdot x},\quad x\in \bb{R}^d.
\end{equation*}
The choice of $\bb{Z}^d$ as index set for $n$ is only a matter of notational convenience; any countable index set would do. 
\begin{Exa}
With $g_n(x):=g(x-na)$ and $b_n:=b$ for all $n\in \bb{Z}^d$ we get
\begin{equation*}
g_{m,n}(x):=g(x-na)e^{2\pi imb\cdot x},\quad x\in \bb{R}^d,
\end{equation*}
which just corresponds to a standard Gabor system.
\end{Exa}
If $\sum_{m,n}|\langle f,g_{m,n}\rangle|^2\asymp \|f\|_2^2$ for all $f\in L^2(\bb{R}^d)$, we refer to $\{g_{m,n}\}_{m,n}$ as an NSGF. For an NSGF $\{g_{m,n}\}_{m,n}$, the frame operator 
\begin{equation*}
Sf=\sum_{m,n\in \bb{Z}^d}\scalarp{f,g_{m,n}}g_{m,n},\quad f\in L^2(\bb{R}^d),
\end{equation*}
is invertible and we have the expansions
\begin{equation*}
f=\sum_{m,n\in \bb{Z}^d}\scalarp{f,g_{m,n}}\tilde{g}_{m,n},\quad  f\in L^2(\bb{R}^d),
\end{equation*}
with $\{\tilde{g}_{m,n}\}_{m,n}:=\{S^{-1}g_{m,n}\}_{m,n}$ being the canonical dual frame of $\{g_{m,n}\}_{m,n}$ \cite{Christensen2016}. For notational convenience we define $G(x):=\sum_{n\in \bb{Z}^d}1/b_n^d\abs{g_n(x)}^2$. With this notation we have the following result \cite[Theorem 1]{Balazs11}.
\begin{The}\label{PAPERC:The:NSGFTheorem}
Let $\{g_n\}_{n\in \bb{Z}^d}\subset L^2(\bb{R}^d)$ with frequency sampling steps $\{b_n\}_{n\in \bb{Z}^d}$,  $b_n>0$ for all $n\in \bb{Z}^d$. Assuming supp$(g_n)\subseteq [0,\frac{1}{b_n}]^d+a_n$, with $a_n\in \bb{R}^d$ for all $n\in \bb{Z}^d$, the frame operator for the system
\begin{equation*}
g_{m,n}(x)=g_n(x)e^{2\pi imb_n\cdot x},\quad \forall m,n\in \bb{Z}^d,\quad x\in \bb{R}^d,
\end{equation*}
is given by
\begin{equation*}
Sf(x)=G(x)f(x),\quad f\in L^2(\bb{R}^d).
\end{equation*}
\begin{sloppypar}
The system $\{g_{m,n}\}_{m,n\in \bb{Z}^d}$ constitutes a frame for $L^2(\bb{R}^d)$, with frame-bounds $0<A\leq B<\infty$, if and only if 
\end{sloppypar}
\begin{equation}\label{PAPERC:eq:nsgtcharacterization}
A\leq G(x)\leq B,\quad \text{for a.e. } x\in \bb{R}^d,
\end{equation}
and the canonical dual frame is then given by
\begin{equation}\label{PAPERC:eq:nsgtdualexpression}
\tilde{g}_{m,n}(x)=\frac{g_{n}(x)}{G(x)}e^{2\pi imb_n\cdot x},\quad x\in \bb{R}^d.
\end{equation}
\end{The}
\begin{Rem}
We note that the canonical dual frame in \eqref{PAPERC:eq:nsgtdualexpression} posses the same structure as the original frame, which is a property not shared by general NSGFs. We also note that the canonical tight frame can be obtained by taking the square root of the denominator in \eqref{PAPERC:eq:nsgtdualexpression}. 
\end{Rem}
Traditionally, an NSGF satisfying the assumptions of \Theref{PAPERC:The:NSGFTheorem} is called a \emph{painless} NSGF, referring to the fact that the frame operator is a simple multiplication operator. This terminology is adopted from the classical \emph{painless nonorthogonal expansions} \cite{Daubechies86}, which corresponds to the painless case for classical Gabor frames. By slight abuse of notation we use the term "painless" to denote the NSGFs satisfying \Defref{PAPERC:Def:PainlessNSGF} below. In order to properly formulate this definition, we first need some preliminary notation which we adopt from \cite{Ottosen2017Article1}.

Let $\{g_n\}_{n\in \bb{Z}^d}\subset L^2(\bb{R}^d)$ satisfy the assumptions in \Theref{PAPERC:The:NSGFTheorem}. Given $C_*>0$ we denote by $\{I_n\}_{n\in \bb{Z}^d}$ the open cubes
\begin{equation}\label{PAPERC:eq:NSGFopencubes}
I_n:=\left(-\varepsilon_n,\frac{1}{b_n}+\varepsilon_n\right)^d+a_n,\quad \forall n\in \bb{Z}^d,
\end{equation}
with $\varepsilon_n:=C_*/b_n$ for all $n\in \bb{Z}^d$. We note that supp$(g_{m,n})\subset I_n$ for all $m,n\in \bb{Z}^d$. For $n\in \bb{Z}^d$ we define
\begin{equation*}
\widetilde{n}:=\left\{n'\in \bb{Z}^d~\big| ~ I_{n'}\cap I_{n}\neq \emptyset\right\},
\end{equation*}
using the notation of \eqref{PAPERC:eq:neighbouringindices}. 
\begin{Def}[Painless NSGF]\label{PAPERC:Def:PainlessNSGF}
Let $\{g_n\}_{n\in \bb{Z}^d}\subset L^2(\bb{R}^d)$ satisfy the assumptions in \Theref{PAPERC:The:NSGFTheorem}, and assume further that,
\begin{enumerate}
\item There exists $C_*>0$ and $n_0\in \bb{N}$, such that the open cubes $\{I_n\}_{n\in \bb{Z}^d}$, given in \eqref{PAPERC:eq:NSGFopencubes}, satisfy $|\widetilde{n}|\leq n_0$ uniformly for all $n\in \bb{Z}^d$.\label{PAPERC:Def:PainlessNSGF1}
\begin{sloppypar}
\item $\{a_n\}_{n\in \bb{Z}^d}$ is a $\delta-$separated set and $\{1+\|a_n\|_2\}_{n\in \bb{Z}^d}$ constitutes a $\{I_n\}_{n\in \bb{Z}^d}-$moderate weight.\label{PAPERC:Def:PainlessNSGF2}
\end{sloppypar}
\item The $g_n$'s are continuous, real valued and satisfy\label{PAPERC:Def:PainlessNSGF3}
\begin{equation*}
g_n(x)\leq C b_n^{d/2}\chi_{I_n}(x),\quad \text{for all }n\in \bb{Z}^d,
\end{equation*}
for some uniform constant $C>0$.
\end{enumerate}
Then we refer to $\{g_{m,n}\}_{m,n\in \bb{Z}^d}$ as a \emph{painless} NSGF.
\end{Def}
The assumptions in \Defref{PAPERC:Def:PainlessNSGF} are easily satisfied, but the support condition in \Theref{PAPERC:The:NSGFTheorem} is rather restrictive and implies a certain redundancy of the system. Nevertheless, we must assume some structure on the dual frame, which is not provided by general NSGFs. We choose the framework of painless NSGFs and base our arguments on the fact that the dual frame possess the same structure as the original frame. We expect it is possible to extend the theory developed in this article to a more general settings by imposing general existence results for NSGFs \cite{Holighaus2014, Dorfler2014,Voigtlaender2016}. We now provide a simple example of a set of window functions satisfying \Defenuref{PAPERC:Def:PainlessNSGF}{PAPERC:Def:PainlessNSGF3}.
\begin{Exa}
Choose a continuous real valued function $\varphi\in L^2(\bb{R}^d)\setminus\{0\}$ with supp$(\varphi)\subseteq [0,1]^d$. For $n\in \bb{Z}^d$ define
\begin{equation*}
g_n(x):=b_n^{d/2}\varphi(b_n(x-a_n)),\quad x\in \bb{R}^d, 
\end{equation*}
with $a_n\in \bb{R}^d$ and $b_n>0$. Then supp$(g_n)\subseteq [0,\frac{1}{b_n}]+a_n$ and \Defenuref{PAPERC:Def:PainlessNSGF}{PAPERC:Def:PainlessNSGF3} is satisfied.
\end{Exa}
Following the approach taken in \cite{Ottosen2017Article1}, we define $Q:=(0,1)^d$ together with the set of affine transformations $\mathcal{T}:=\{A_n(\cdot)+c_n\}_{n\in \bb{Z}^d}$ with
\begin{equation*}
A_n:=\left(2\varepsilon_n+\frac{1}{b_n}\right)\cdot I_d,\quad\text{and}\quad (c_n)_j:=-\varepsilon_n+(a_n)_j,\quad 1\leq j\leq d.
\end{equation*}
It is then easily shown that $\mathcal{Q}:=\{Q_T\}_{T\in \mathcal{T}}=\{I_n\}_{n\in \bb{Z}^d}$ forms a structured covering of $\bb{R}^d$ \cite[Lemma 4.1]{Ottosen2017Article1}. Given $s\in \bb{R}$ and $1\leq p,q\leq \infty$, we may therefore construct the associated decomposition space $D(\mathcal{Q},L^p,\ell^q_{\omega^s})$ with $\{\omega_T\}_{T\in \mathcal{T}}:=\{1+\|a_n\|_2\}_{n\in \bb{Z}^d}$. 
\begin{Exa}
Let $\{g_{m,n}\}_{m\in \bb{Z}^d,n\in \bb{Z}^d}$ be a painless NSGF according to \Defref{PAPERC:Def:PainlessNSGF}. Assume additionally that $K:=\inf\{b_n\}_{n\in \bb{Z}^d}>0$ and that \Defenuref{PAPERC:Def:PainlessNSGF}{PAPERC:Def:PainlessNSGF1} and \Defenuref{PAPERC:Def:PainlessNSGF}{PAPERC:Def:PainlessNSGF2}  hold for the larger cubes $K_n:=(-\varepsilon,1/K+\varepsilon)^d+a_n$ for some $\varepsilon>0$. Defining $Q:=(0,1)^d$ and $\mathcal{T}:=\{A_n(\cdot)+c_n\}_{n\in \bb{Z}^d}$, with
\begin{equation*}
A_n:=\left(2\varepsilon+\frac{1}{K}\right)\cdot I_d,\quad\text{and}\quad (c_n)_j:=-\varepsilon+(a_n)_j,\quad 1\leq j\leq d,
\end{equation*}
we obtain the structured covering $\mathcal{Q}:=\{K_n\}_{n\in \bb{Z}^d}$. In this special case the associated decomposition space is the Wiener amalgam space $W(L^p,\ell^q_{\omega^s})$ for $s\in \bb{R}$ and $1\leq p,q\leq \infty$ (cf. \Exaref{PAPERC:Exa:wieneramalgam}).
\end{Exa}
For the rest of this article, we write $\{g_{m,T}\}_{m\in \bb{Z}^d,T\in \mathcal{T}}$ for a painless NSGF with associated structured covering $\mathcal{Q}:=\{Q_T\}_{T\in \mathcal{T}}$. With this notation, then supp$(g_{m,T})\subset Q_T$ for all $m\in \bb{Z}^d$ and all $T\in \mathcal{T}$. Similarly we write $\{\omega_T\}_{T\in \mathcal{T}}=\{1+\|a_T\|_2\}_{T\in \mathcal{T}}$ for the associated weight function.

\section{Characterization of decomposition spaces}\label{PAPERC:Sec:4}
Using the notation of \cite{Borup07} we define the sequence space $d(\mathcal{Q},\ell^p,\ell^q_{\omega^s})$ as the set of coefficients $\{c_{m,T}\}_{m\in \bb{Z}^d,T\in \mathcal{T}}\subset \bb{C}$ satisfying
\begin{equation*}
\norm{\{c_{m,T}\}_{m\in \bb{Z}^d,T\in \mathcal{T}}}_{d(\mathcal{Q},\ell^p,\ell^q_{\omega^s})}:=\norm{\left\{\norm{\left\{c_{m,T}\right\}_{m\in \bb{Z}^d}}_{\ell^p}\right\}_{T\in \mathcal{T}}}_{\ell^q_{\omega^s}}<\infty,
\end{equation*}
for $s\in \bb{R}$ and $1\leq p,q\leq\infty$. We can now prove the following important stability result.
\begin{The}\label{PAPERC:The:characterization}
Let $\{g_{m,T}\}_{m\in \bb{Z}^d,T\in \mathcal{T}}$ be a painless NSGF with associated structured covering $\mathcal{Q}=\{Q_T\}_{T\in \mathcal{T}}$ and weight function $\{\omega_T\}_{T\in \mathcal{T}}=\{1+\|a_T\|_2\}_{T\in \mathcal{T}}$. Fix $s\in \bb{R}$, $1\leq p\leq 2$ and let $p':=p/(p-1)$. For $f\in D(\mathcal{Q},L^p,\ell^q_{\omega^s})$ and $1\leq q\leq \infty$, 
\begin{equation}
\norm{\left\{\scalarp{f,g_{m,T}}\right\}_{m\in \bb{Z}^d,T\in \mathcal{T}}}_{d(\mathcal{Q},\ell^{p'},\ell^q_{\omega^s})}\leq C\norm{f}_{D(\mathcal{Q},L^p,\ell^q_{\omega^s})},\label{PAPERC:eq:characterization1}
\end{equation}
and for $h\in D(\mathcal{Q},L^{p'},\ell^q_{\omega^s})$ and $1\leq q< \infty$, 
\begin{equation}
\norm{h}_{D(\mathcal{Q},L^{p'},\ell^q_{\omega^s})}\leq C' \norm{\left\{\scalarp{h,g_{m,T}}\right\}_{m\in \bb{Z}^d,T\in \mathcal{T}}}_{d(\mathcal{Q},\ell^p,\ell^q_{\omega^s})}.\label{PAPERC:eq:characterization2}
\end{equation}
\end{The}
\begin{proof}
We first prove \eqref{PAPERC:eq:characterization1}. Given $f\in D(\mathcal{Q},L^p,\ell^q_{\omega^s})$, since $\widetilde{\psi_T}:=\sum_{T'\in \widetilde{T}}\psi_T\equiv 1$ on $Q_T$, then
\begin{align*}
&\norm{\{\scalarp{f,g_{m,T}}\}_{m\in \bb{Z}^d}}_{\ell^{p'}}=\left(\sum_{m\in \bb{Z}^d}\abs{\scalarp{\widetilde{\psi_T}f,g_{m,T}}}^{p'}\right)^{1/p'}\\
&=b_T^{-d/2}\left(\sum_{m\in \bb{Z}^d}\abs{b_T^{d/2}\int_{\bb{R}^d}\widetilde{\psi_T}(x)f(x)g_{T}(x)e^{-2\pi i mb_T\cdot x}dx}^{p'}\right)^{1/p'},
\end{align*}
with $b_T>0$ being the frequency sampling step. Since $1\leq p\leq 2$ we can use the Hausdorff-Young inequality \cite[Theorem 2.1 on page 98]{Katznelson1976}, which together with \Defenuref{PAPERC:Def:PainlessNSGF}{PAPERC:Def:PainlessNSGF3} imply
\begin{equation*}
\norm{\{\scalarp{f,g_{m,T}}\}_{m\in \bb{Z}^d}}_{\ell^{p'}}\leq b_T^{-d/{2}}\norm{\widetilde{\psi_T}fg_T}_{L^p}\leq C_1\norm{\widetilde{\psi_T}f}_{L^p}.
\end{equation*}
Hence, using \eqref{PAPERC:eq:Operatornorm} we get
\begin{align*}
\norm{\left\{\scalarp{f,g_{m,T}}\right\}_{m\in \bb{Z}^d,T\in \mathcal{T}}}_{d(\mathcal{Q},\ell^{p'},\ell^q_{\omega^s})}&\leq C_1\norm{\left\{\norm{\widetilde{\psi_T}f}_{L^p}\right\}_{T\in \mathcal{T}}}_{\ell^q_{\omega^s}}\\
&\leq C_2\norm{f}_{D(\mathcal{Q},L^p,\ell^q_{\omega^s})}.
\end{align*}
Let us now prove \eqref{PAPERC:eq:characterization2}. Given $h\in D(\mathcal{Q},L^{p'},\ell^q_{\omega^s})$ we may write the norm as
\begin{equation}\label{PAPERC:eq:characterizationproof0}
\norm{h}_{D(\mathcal{Q},L^{p'},\ell^q_{\omega^s})}=\sup_{\sigma\in \mathcal{S}(\bb{R}^d),\norm{\sigma}_{D(\mathcal{Q},L^{p},\ell^{q'}_{\omega^{-s}})}=1}\abs{\scalarp{h,\sigma}},\qquad q':=q/(q-1),
\end{equation}
since the dual space of $D(\mathcal{Q},L^{p'},\ell^q_{\omega^s})$ can be identified with $D(\mathcal{Q},L^p,\ell^{q'}_{\omega^{-s}})$ and since $\mathcal{S}(\bb{R}^d)$ is dense in $D(\mathcal{Q},L^p,\ell^{q'}_{\omega^{-s}})$. Given $\sigma\in \mathcal{S}(\bb{R}^d)$, with $\|\sigma\|_{D(\mathcal{Q},L^{p},\ell^{q'}_{\omega^{-s}})}=1$, we write the frame expansion of $\sigma$ with respect to $\{g_{m,T}\}_{m,T}$ and apply H\"{o}lder's inequality twice to obtain
\begin{align}\label{PAPERC:eq:characterizationproof1}
\abs{\scalarp{h,\sigma}}&\leq \sum_{T\in \mathcal{T}}\sum_{m\in \bb{Z}^d}\abs{\scalarp{\sigma,g_{m,T}}\scalarp{h,\tilde{g}_{m,T}}}\notag\\
&\leq \sum_{T\in \mathcal{T}}\norm{\{\scalarp{\sigma,g_{m,T}}\}_{m\in \bb{Z}^d}}_{\ell^{p'}}\norm{\{\scalarp{h,\tilde{g}_{m,T}}\}_{m\in \bb{Z}^d}}_{\ell^{p}}\notag\\
&\leq \norm{\left\{\scalarp{\sigma,g_{m,T}}\right\}_{m\in \bb{Z}^d,T\in \mathcal{T}}}_{d(\mathcal{Q},\ell^{p'},\ell^{q'}_{\omega^{-s}})} \norm{\left\{\scalarp{h,\tilde{g}_{m,T}}\right\}_{m\in \bb{Z}^d,T\in \mathcal{T}}}_{d(\mathcal{Q},\ell^p,\ell^q_{\omega^{s}})}.
\end{align}
According to \eqref{PAPERC:eq:characterization1} then 
\begin{equation*}
\norm{\{\scalarp{\sigma,g_{m,T}}\}_{m,T}}_{d(\mathcal{Q},\ell^{p'},\ell^{q'}_{\omega^{-s}})}\leq C_1\norm{\sigma}_{D(\mathcal{Q},L^p,\ell^{q'}_{\omega^{-s}})}=C_1,
\end{equation*}
which combined with \eqref{PAPERC:eq:characterizationproof1} and \eqref{PAPERC:eq:nsgtcharacterization} yield
\begin{align}\label{PAPERC:eq:characterizationproof2}
\abs{\scalarp{h,\sigma}}&\leq C_1\norm{\left\{\scalarp{h,\tilde{g}_{m,T}}\right\}_{m\in \bb{Z}^d,T\in \mathcal{T}}}_{d(\mathcal{Q},\ell^p,\ell^q_{\omega^{s}})} \notag \\
&\leq C_2\norm{\left\{\scalarp{h,g_{m,T}}\right\}_{m\in \bb{Z}^d,T\in \mathcal{T}}}_{d(\mathcal{Q},\ell^p,\ell^q_{\omega^{s}})},
\end{align}
with $C_2:=C_1/A$. Finally, combining \eqref{PAPERC:eq:characterizationproof0} and \eqref{PAPERC:eq:characterizationproof2} we arrive at
\begin{equation*}
\norm{h}_{D(\mathcal{Q},L^{p'},\ell^q_{\omega^s})}\leq C_2\norm{\left\{\scalarp{h,g_{m,T}}\right\}_{m\in \bb{Z}^d,T\in \mathcal{T}}}_{d(\mathcal{Q},\ell^p,\ell^q_{\omega^{s}})},
\end{equation*}
which proves \eqref{PAPERC:eq:characterization2}.
\end{proof}
We note that for $s\in \bb{R}$, $1\leq q<\infty$ and $p=2$, \Theref{PAPERC:The:characterization} yields the equivalence	
\begin{equation*}
\norm{f}_{D(\mathcal{Q},L^2,\ell^q_{\omega^s})}\asymp \norm{\left\{\scalarp{f,g_{m,T}}\right\}_{m\in \bb{Z}^d,T\in \mathcal{T}}}_{d(\mathcal{Q},\ell^2,\ell^q_{\omega^s})},\quad f\in D(\mathcal{Q},L^2,\ell^q_{\omega^s}).
\end{equation*}
It follows that the \emph{coefficient operator} $C:f\rightarrow \{\langle f,g_{m,T}\rangle\}_{m,T}$ is bounded from $D(\mathcal{Q},L^2,\ell^q_{\omega^s})$ into $d(\mathcal{Q},\ell^2,\ell^q_{\omega^s})$. We define the corresponding \emph{reconstruction operator} as
\begin{equation*}
R\left(\{c_{m,T}\}_{m\in \bb{Z}^d,T\in \mathcal{T}}\right)=\sum_{T\in \mathcal{T}}\sum_{m\in \bb{Z}^d}c_{m,T}\tilde{g}_{m,T},\qquad \forall \{c_{m,T}\}_{m,T}\in d(\mathcal{Q},\ell^2,\ell^q_{\omega^s}).
\end{equation*}
With this notation we have the following result.
\begin{Pro}\label{PAPERC:Pro:CharacterizationFor2}
Let $\{g_{m,T}\}_{m\in \bb{Z}^d,T\in \mathcal{T}}$ be a painless NSGF with associated structured covering $\mathcal{Q}=\{Q_T\}_{T\in \mathcal{T}}$ and weight function $\{\omega_T\}_{T\in \mathcal{T}}=\{1+\|a_T\|_2\}_{T\in \mathcal{T}}$. Given $s\in \bb{R}$ and $1\leq q<\infty$, the reconstruction operator $R$ is bounded from $d(\mathcal{Q},\ell^2,\ell^q_{\omega^s})$ onto $D(\mathcal{Q},L^2,\ell^q_{\omega^s})$ and we have the expansions
\begin{equation}\label{PAPERC:eq:unconditionalconvergence}
f=RC(f)=\sum_{m\in \bb{Z}^d,T\in \mathcal{T}}\scalarp{f,g_{m,T}}\tilde{g}_{m,T},\quad f\in D(\mathcal{Q},L^2,\ell^q_{\omega^s}),
\end{equation}
with unconditional convergence.
\end{Pro}
\begin{proof}
We first prove that $R$ is bounded. Given $\{c_{m,T}\}_{m,T}\in d(\mathcal{Q},\ell^2,\ell^q_{\omega^s})$, \eqref{PAPERC:eq:Operatornorm} and \eqref{PAPERC:eq:nsgtcharacterization} yield
\begin{align}\label{PAPERC:eq:boundedrecon1}
\norm{R(\{c_{m,T}\}_{m,T})}_{D(\mathcal{Q},L^2,\ell^q_{\omega^s})}&=\norm{\left\{\norm{\psi_T \left(\sum_{T'\in \widetilde{T}}\sum_{m\in \bb{Z}^d}c_{m,T'}\tilde{g}_{m,T'}\right)}_{L^2}\right\}_{T\in \mathcal{T}}}_{\ell^q_{\omega^s}}\notag \\
&\leq C_1\norm{\left\{\norm{\sum_{m\in \bb{Z}^d}c_{m,T}g_{m,T}}_{L^2}\right\}_{T\in \mathcal{T}}}_{\ell^q_{\omega^s}}.
\end{align}
Applying \Defenuref{PAPERC:Def:PainlessNSGF}{PAPERC:Def:PainlessNSGF3} and the Hausdorff-Young inequality \cite[Theorem 2.2 on page 99]{Katznelson1976} we get
\begin{align}\label{PAPERC:eq:boundedrecon2}
\norm{\sum_{m\in \bb{Z}^d}c_{m,T}g_{m,T}}_{L^2}^2\leq C\int_{\bb{R}^d}\abs{b_T^{d/2}\sum_{m\in \bb{Z}^d}c_{m,T}e^{2\pi i mb_T\cdot x}}^2dx\leq C\norm{\left\{c_{m,T}\right\}_{m\in \bb{Z}^d}}_{\ell^2}^2.
\end{align}
Combining \eqref{PAPERC:eq:boundedrecon1} and \eqref{PAPERC:eq:boundedrecon2} we arrive at
\begin{align}\label{PAPERC:eq:boundedrecoperator}
\norm{R(\{c_{m,T}\}_{m,T})}_{D(\mathcal{Q},L^2,\ell^q_{\omega^s})}&\leq C_2\norm{\left\{\norm{\left\{c_{m,T}\right\}_{m\in \bb{Z}^d}}_{\ell^2}\right\}_{T\in \mathcal{T}}}_{\ell^q_{\omega^s}}\notag\\
&=C_2\norm{\{c_{m,T}\}_{m,T}}_{d(\mathcal{Q},\ell^2,\ell^q_{\omega^s})},
\end{align}
which shows the boundedness of $R$. Let us now prove the unconditional convergence of \eqref{PAPERC:eq:unconditionalconvergence}. Given $f\in D(\mathcal{Q},L^2,\ell^q_{\omega^s})$ we can find a sequence $\{f_k\}_{k\in \bb{N}}\subset \mathcal{S}(\bb{R}^d)$ such that $f_k\rightarrow f$ in $D(\mathcal{Q},L^2,\ell^q_{\omega^s})$. For each $k$ we have the expansion $f_k=RC(f_k)$ and by continuity of $RC$ we get $f=RC(f)$. Given $\varepsilon>0$, \eqref{PAPERC:eq:boundedrecoperator} implies that we can find a \emph{finite} subset $F_0\subseteq \bb{Z}^d\times \mathcal{T}$, such that for all finite sets $F\supseteq F_0$,
\begin{equation*}
\norm{f-\sum_{(m,T)\in F}\scalarp{f,g_{m,T}}\tilde{g}_{m,T}}_{D(\mathcal{Q},L^2,\ell^q_{\omega^s})}\leq C_2 \norm{\{\scalarp{f,g_{m,T}}\}_{(m,T)\notin F}}_{d(\mathcal{Q},\ell^2,\ell^q_{\omega^s})}<\varepsilon.
\end{equation*}
According to \cite[Proposition 5.3.1 on page 98]{Grochenig2001}, this property is equivalent to unconditional convergence.
\end{proof}
Based on \Proref{PAPERC:Pro:CharacterizationFor2}, we can show some important properties of $\{g_{m,T}\}_{m,T}$ in connection with nonlinear approximation theory \cite{Devore1993,DeVore98}. Assume $f\in D(\mathcal{Q},L^2,\ell^2_{\omega^s})$, for $s\in \bb{R}$, and write the frame expansion
\begin{equation}\label{PAPERC:eq:ApplicationNonlinear1}
f=\sum_{m\in \bb{Z}^d,T\in \mathcal{T}}\scalarp{f,g_{m,T}}\tilde{g}_{m,T}.
\end{equation}
Let $\{\theta_k\}_{k\in \bb{N}}$ be a rearrangement of the frame coefficients $\{\langle f,g_{m,T}\rangle\}_{m,T}$ such that $\{|\theta_k|\}_{k\in \bb{N}}$ constitutes a non-increasing sequence. Also, let $f_N$ be the $N$-term approximation to $f$ obtained by extracting the terms in \eqref{PAPERC:eq:ApplicationNonlinear1} corresponding to the $N$ largest coefficients $\{\theta_k\}_{k=1}^N$. Since $R$ is bounded, \cite[Theorem 6]{Gribonval2004} implies that for each $1\leq \tau<2$,
\begin{align}\label{PAPERC:eq:ApplicationNonlinear3}
\norm{f-f_N}_{D(\mathcal{Q},L^2,\ell^2_{\omega^s})}&\leq C_1\norm{\left\{\theta_k\right\}_{k>N}}_{d(\mathcal{Q},\ell^2,\ell^2_{\omega^s})}\leq C_2N^{-\alpha}\norm{\left\{\theta_k\right\}_{k\in \bb{N}}}_{d(\mathcal{Q},\ell^\tau,\ell^\tau_{\omega^s})}\notag\\
&= C_2N^{-\alpha}\norm{\left\{\scalarp{f,g_{m,T}}\right\}_{k\in \bb{N}}}_{d(\mathcal{Q},\ell^\tau,\ell^\tau_{\omega^s})},\quad \alpha:=1/\tau-1/2.
\end{align}
We conclude that for $f\in D(\mathcal{Q},L^2,\ell^2_{\omega^s})$, with frame coefficients in $d(\mathcal{Q},\ell^\tau,\ell^\tau_{\omega^s})$, we obtain good approximations in $D(\mathcal{Q},L^2,\ell^2_{\omega^s})$ by thresholding the frame coefficients in \eqref{PAPERC:eq:ApplicationNonlinear1}. The rate of the approximation is given by $\alpha\in (0,1/2]$.

\section{Numerical experiments}\label{PAPERC:Sec:5}
In this section we provide the numerical experiments, thresholding coefficients of both stationary and nonstationary Gabor expansions. We note that analyzis with a stationary Gabor frame corresponds to analyzis with the short-time Fourier transform (STFT) as the Gabor coefficients can be re-written as
\begin{equation*}
\scalarp{f,g_{m,n}}=\int_{\bb{R}^d}f(t)\overline{g(t-na)}e^{-2\pi imb\cdot t}dt=V_gf(na,mb),\quad f\in L^2(\bb{R}^d),
\end{equation*}
with $V_gf(na,mb)$ denoting the STFT of $f$, with respect to $g$, at time $na$ and frequency $mb$. 

For the implementation we use MATLAB 2017B and in particular we use the following two toolboxes: The LTFAT \cite{ltfatnote030} (version 2.2.0 or above) available from \url{http://ltfat.github.io/} and the NSGToolbox \cite{Balazs11} (version 0.1.0 or above) available from \url{http://nsg.sourceforge.net/}. The sound files we consider are part of the EBU-SQAM database \cite{EBUSQAM}, which consists of 70 test sounds sampled at 44.1 kHz. The test sounds form a large variety of speech and music including single instruments, classical orchestra, and pop music. Since music signals are continuous signals of finite energy, it make sense to consider them in the framework of decomposition spaces. Moreover, the decomposition space norm constitutes a natural measure for such nonstationary signals, capable of detecting local signal changes as opposed to the standard $L^p-$norm.

We divide the numerical analysis into two sections. In \Secref{PAPERC:Sec:SingleExperiment} we compare the performance of an adaptive nonstationary Gabor expansion to that of a classical Gabor expansion by analyzing spectrograms, reconstruction errors, and approximation rates associated to a particular music signal (signal 39 of the EBU-SQAM database). Then, in \Secref{PAPERC:Sec:Largescaleexperiment} we extend the experiment to cover the entire EBU-SQAM database and compare the average reconstruction errors and approximation rates, taken over the 70 test signals, for the two methods. To analyse the performance of an expansion we use the relative root mean square (RMS) reconstruction error
\begin{equation*}
\text{RMS}(f,f_{rec}):=\frac{\norm{f-f_{rec}}_2}{\norm{f}_2}.
\end{equation*}
As a general rule of thumb, an RMS error below $1\%$ is hardly noticeable to the average listener. We measure the redundancy of a transform by
\begin{equation*}
\frac{\text{number of coefficients}}{\text{length of signal}}.
\end{equation*}
The redundancy of the adaptive NSGF is approximately $5/3$ and we have chosen parameters for the stationary Gabor frame, which mathes this redundancy.

\subsection{Single experiment}\label{PAPERC:Sec:SingleExperiment}
In this experiment we consider sample 22000-284143 of signal 39 in the EBU-SQAM database. This signal is a piece of piano music consisting of an increasing melody of 10 individual tones (taken from an F major chord) starting at F2 (87 Hz fundamental frequency) and ending at F5 (698 Hz fundamental frequency). We construct the Gabor expansion using 1536 frequency channels and a hop size of 1024. The window function is chosen as a Hanning window of length 1536 such that the resulting system constitutes a painless Gabor frame. The Gabor transform has a redundancy of $\approx 1.51$ and the total number of Gabor coefficients is $198402$ (of which $195326$ are non-zero). We only work with the coefficients of the positive frequencies since the signal is real valued. Performing hard thresholding, and keeping only the $15800$ largest coefficients, we obtain a reconstructed signal with an RMS reconstruction error just below $1\%$. 

For the adaptive NSGF, we choose to follow the adaptation procedure from \cite{Balazs11}, resulting in the construction of so-called \emph{scale frames}. The idea is to calculate the onsets of the music piece, using a separate algorithm \cite{Dixon06}, and then to use short window functions around the onsets and long window functions between the onsets. The space between two onsets is spanned in such a way that the window length first increases (as we move away from the first onset) and then decreases (as we approach the second onset). To obtain a smooth resolution, the construction is such that adjacent windows are either of the same length or one is twice as long as the other. We refer the reader to \cite{Balazs11} for further details. For the actual implementation, we use 8 different Hanning windows with lengths varying from $192$ (around the onsets) to $192\cdot2^7=24576$. For the particular signal, the nonstationary Gabor transform has a redundancy of $\approx 1.66$, which is comparable to that of the Gabor transform. The total number of coefficients is $217993$ (of which $216067$ are non-zero). Again, we only consider the coefficients of the positive frequencies. Keeping the $13100$ largest coefficients we obtain an expansion with an RMS reconstruction error just below $1\%$. This is considerably fewer coefficients than needed for the stationary Gabor expansion, which shows a natural sparseness of scale frames for this particular signal class. This property was already noted by the authors in \cite{Balazs11}. Spectrograms based on the original expansions and the thresholded expansions can be found in \Figref{PAPERC:fig:FourSpectrograms}.
\begin{figure}[ht]
\center
\includegraphics[width=\textwidth]{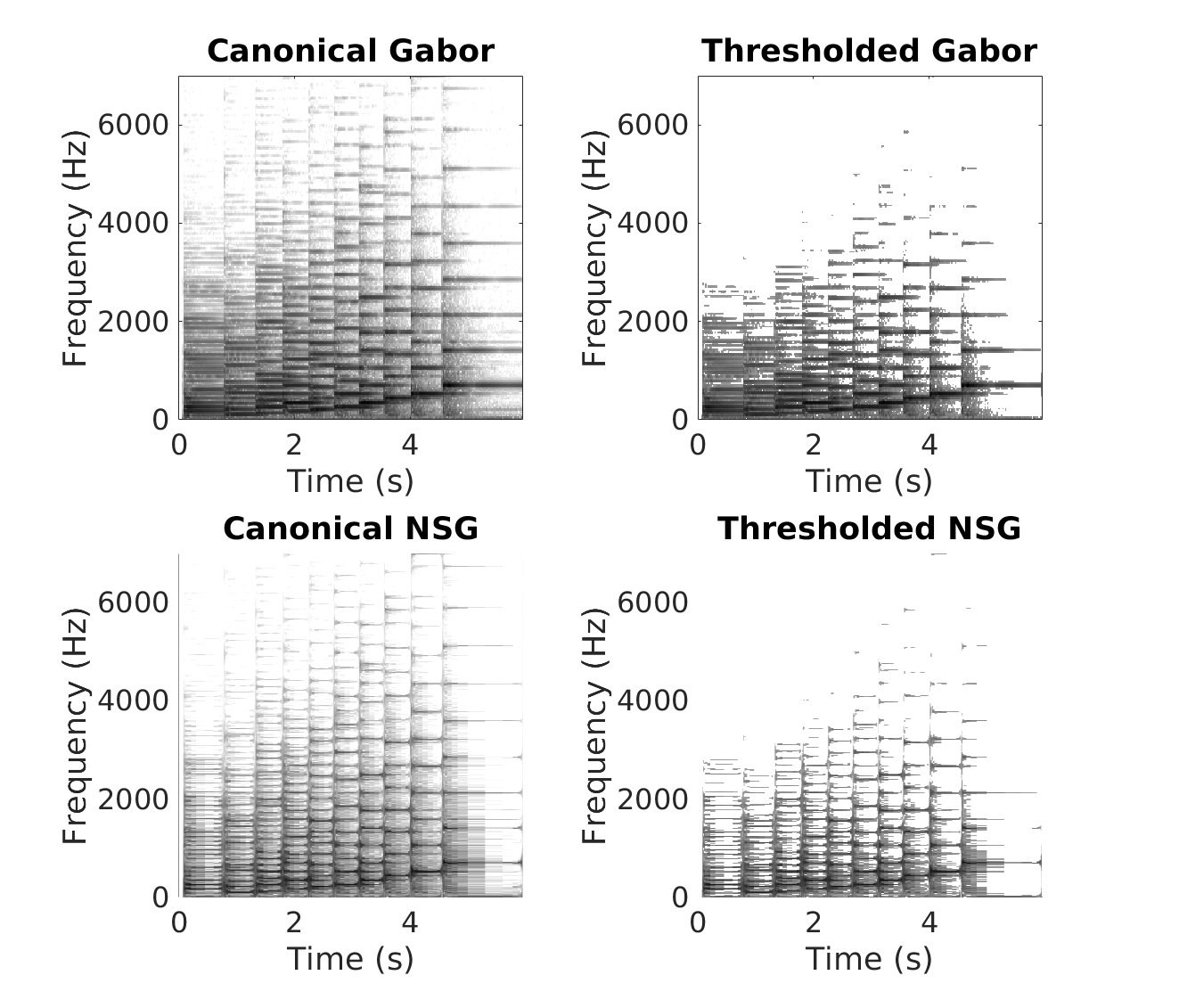}
\caption{Spectrograms based on the original and thresholded Gabor- and nonstationary Gabor (NSG) expansions with RMS errors just below $1\%$.}
\label{PAPERC:fig:FourSpectrograms}      
\end{figure}

The 10 "vertical stripes" in the spectrograms correspond to the onsets of the 10 tones in the melody and the "horizontal stripes" correspond to the frequencies of the harmonics. We note that the adaptive behaviour of the NSGF is clearly visible in the spectrograms, resulting in a good time resolution around the onsets and a good frequency resolution between the onsets. In contrast to this behaviour, the stationary Gabor frame uses a uniform resolution over the whole time-frequency plane. 

Based on the results from \Secref{PAPERC:Sec:4} (in particular \eqref{PAPERC:eq:ApplicationNonlinear3}), we expect the RMS error $E(N)$ to decrease as $N^{-\alpha}$, for some $\alpha>0$, with $N$ being the number of non-zero coefficients. Calculating $E(N)$ for different values of $N$ and performing power regression, we obtain the plots shown in \Figref{PAPERC:Fig:ApproxRate}.
\begin{figure}[ht]
\center
\includegraphics[width=\textwidth]{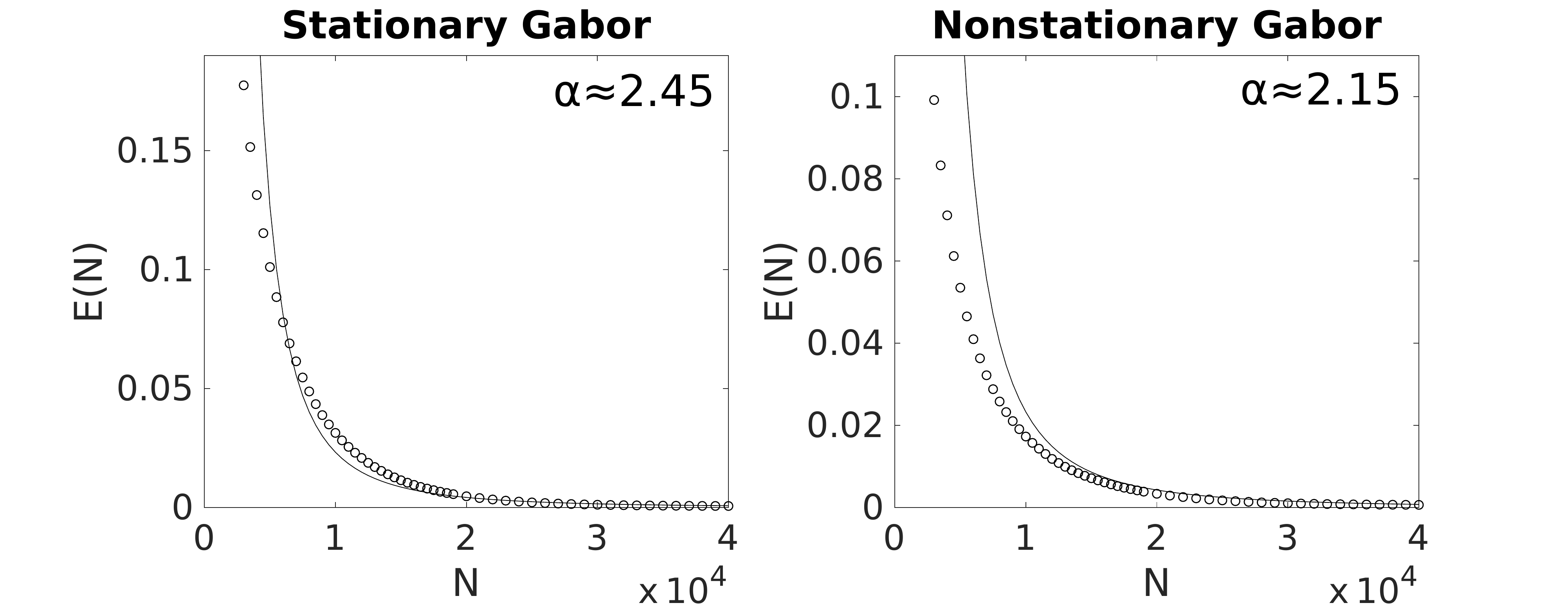}
\caption{RMS error $E(N)$ as a function of $N$, the number of non-zero coefficients, for both stationary and nonstationary Gabor expansions. Also, an estimated power function is plotted for each expansion together with the associated value of the exponent $\alpha$.}
\label{PAPERC:Fig:ApproxRate}      
\end{figure}

The results in \Figref{PAPERC:Fig:ApproxRate} show that both the RMS error $E(N)$ and the approximation rate $\alpha$ are lower for the nonstationary Gabor expansion than for the stationary Gabor expansion. Clearly, a small RMS error is more important than a fast approximation rate. Also, the fast approximation rate for the stationary Gabor frame is caused mainly by the high RMS error associated with small values of $N$. We note that both approximation rates are considerably faster than the rate given in \eqref{PAPERC:eq:ApplicationNonlinear3} (which belongs to $(0,1/2]$). This illustrates that \eqref{PAPERC:eq:ApplicationNonlinear3} only provides us with an \emph{upper bound} on the approximation error --- the actual error might be much smaller. It also illustrates that both methods work extremely well for this kind of sparse signal. In the next section we extend the analyzis presented here to cover the entire EBU-SQAM database.

\subsection{Large scale experiment}\label{PAPERC:Sec:Largescaleexperiment}
For this experiment we consider the first $524288$ samples of each of the $70$ test sounds avaliable in the EBU-SQAM database. For each test sound we construct a nonstationary Gabor expansion, with parameters as described in \Secref{PAPERC:Sec:SingleExperiment}, and three stationary Gabor expansions with different parameter settings. Using the notation (hopsize,number of frequency channels), we use the parameter settings $(1024,2048)$, $(1536,2048)$, and $(1024,1536)$ for the three Gabor expansions. The window function associated to a Gabor expansion is chosen as a Hanning window with length equal to the corresponding number of frequency channels (resulting in a painless Gabor frame). For each of the four expansions we calculate for each test sound:\newpage
\begin{enumerate}
\item The redundancy of the (non-thresholded) expansion.

\item Thresholded expansions with respect to $N$, the number of non-zero coefficient, where $N$ takes on the values
\begin{equation*}
N\in \left\{10000, 11000, \cdots, 29000, 30000, 35000, \cdots , 195000, 200000\right\}.
\end{equation*}

\item The sum of RMS errors $\sum_N E(N)$ taken over all $55$ possible values of $N$.

\item The value $\alpha$ of the estimated power function.
\end{enumerate}
Repeating the experiment for all 70 test sounds we get the averaged values shown in Table \ref{PAPERC:tab:AverageResults}.

\begin{table}[h!]
\centering
\caption{Average redundancies, sum of RMS errors, and approximation rates taken over the $70$ test signals in the EBU-SQAM database. The experiment includes three stationary Gabor frames, with different parameters settings, and one NSGF.}
\label{PAPERC:tab:AverageResults}
\begin{tabular}{l c c c c}
\hline
Transform: & G$(1024,2048)$ & G$(1536,2048)$ & G$(1024,1536)$ & NSGF\\
\hline
Average redun.: & $2.0020$ & $1.3451$ & $1.5049$ & $1.6206$\\
Average error: & $2.1448$ & $1.9128$ & $1.9492$ & $1.7367$\\
Average $\alpha$: & $1.3088$ & $1.4455$ & $1.4278$ & $1.2606$\\
\hline
\end{tabular}
\end{table}

The results in Table \ref{PAPERC:tab:AverageResults} show the same behaviour as the experiment in \Secref{PAPERC:Sec:SingleExperiment} --- The NSGF provides the smallest RMS error and the slowest approximation rate. We note that the approximation rates all belong to the interval $[1.25;1.45]$, which is much lower than the rates obtained in \Secref{PAPERC:Sec:SingleExperiment}. This is due to the fact that the piano signal in \Secref{PAPERC:Sec:SingleExperiment} has a very sparse expansion, which is not true for all $70$ test signals in the database. At first glance, the Gabor frame which seems to provide the best results is the one with parameter settings $(1536,2048)$ --- it produces the smallest RMS error and the largest approximation rate. However, this is mainly due to the low redundancy of the frame, which is only around $1.35$. A low redundancy implies fewer Gabor coefficients (with more time-frequency information contained in each coefficient), which implies good results in terms of RMS error and approximation rate. However, a low redundancy also implies a worsened time-frequency resolution, which is not desirable for practical purposes. Finally, it is worth noting that the NSGF produces a significantly lower RMS error than the Gabor frame with parameters $(1536,2048)$ even with a higher redundancy.

\section{Conclusion}\label{PAPERC:Sec:6}
We have provided a self-contained description of decomposition spaces on the time side and proven several important properties of such spaces. Given a painless NSGF with flexible time resolution, we have shown how to construct an associated decomposition space, which characterizes signals with sparse expansions relative to the NSGF. Based on this characterization we have proven an upper bound on the approximation error occurring when thresholding the coefficients of the frame expansions. The theoretical results have been complemented with numerical experiments, illustrating that the approximation error is indeed smaller than the theoretical upper bound. Using terminology from nonlinear approximation theory, we have proven a Jackson inequality for nonlinear approximation with certain NSGFs. It could be interesting to consider the inverse estimate, a so-called Bernstein inequality, providing us with a lower bound on the approximation error. The numerical experiments indeed suggest that the approximation error acts as a power function of the number of non-zero coefficients. Unfortunately, obtaining a Bernstein inequality for such a redundant dictionary is in general beyond the reach of current methods \cite{Gribonval2006}.

\appendix
\section{Proof of \Theref{PAPERC:The:decomcovthe}}\label{PAPERC:App:1}
\begin{proof}
We will use the well known fact that
\begin{equation}\label{PAPERC:eq:wellknown1}
\int_{\bb{R}^d}(1+\norm{x}_2)^{-m}dx<\infty,\quad m>d.
\end{equation}
We prove each of the four statements separately and we write $D^s_{p,q}:=D(\mathcal{Q},L^p,\ell^q_{\omega^s})$ to simplify notation.
\begin{enumerate}
\item Repeating the arguments from \cite[Proposition 5.7]{Borup08}, using \Defref{PAPERC:Def:SAC}\eqref{PAPERC:Def:SAC2}, we can show that
\begin{equation}\label{PAPERC:eq:decomcovtheproof1}
D^{s+\varepsilon}_{p,\infty}\hookrightarrow D^s_{p,q}\hookrightarrow D^s_{p,\infty},\quad \varepsilon>d/q,
\end{equation}
for any $s\in \bb{R}$ and $1\leq p,q \leq \infty$. Hence, to prove \Theenuref{PAPERC:The:decomcovthe}{PAPERC:The:decomcovthe1} it suffices to show that $\mathcal{S}(\bb{R}^d)\hookrightarrow D^s_{p,\infty} \hookrightarrow \mathcal{S}'(\bb{R}^d)$ for any $s\in \bb{R}$ and $1\leq p\leq\infty$. We first show that $\mathcal{S}(\bb{R}^d)\hookrightarrow D^s_{p,\infty}$. Since $\{\omega_T\}_{T\in \mathcal{T}}=\{1+\|x_T\|_2\}_{T\in \mathcal{T}}$ is $\mathcal{Q}-$moderate, and $\psi_T$ is uniformly bounded, this result follows from \eqref{PAPERC:eq:wellknown1} since
\begin{align*}
\omega_T^s\norm{\psi_T f}_{L^p}&\leq C_1 \norm{(1+\norm{\cdot}_2)^s\psi_T f}_{L^p}\leq C_1 \norm{(1+\norm{\cdot}_2)^s f}_{L^p}\\
&\leq C_2 \norm{(1+\norm{\cdot}_2)^{s+r} f}_{L^\infty}\\
&\leq C_2\max_{|\beta|\leq N}\sup_{x\in \bb{R}^d}\abs{(1+\norm{x}_2)^N\partial^\beta_x f(x)},\quad f\in \mathcal{S}(\bb{R}^d),
\end{align*}
for $r>d/p$ and $N\geq s+r$. To show that $D^s_{p,\infty} \hookrightarrow \mathcal{S}'(\bb{R}^d)$, we define $\widetilde{\psi_T}:=\sum_{T'\in \widetilde{T}}\psi_{T'}$. Given $f\in D^s_{p,\infty}$ and $\varphi\in \mathcal{S}(\bb{R}^d)$, H\"{o}lder's inequality yields
\begin{align}\label{PAPERC:eq:embeddinglargep}
\abs{\scalarp{f,\varphi}}&=\abs{\sum_{T\in \mathcal{T}}\scalarp{\psi_T f,\widetilde{\psi_T}\varphi}}\leq \sum_{T\in \mathcal{T}} \norm{\psi_T f \widetilde{\psi_T}\varphi}_{L^1}\notag\\
&\leq \sum_{T\in \mathcal{T}} \norm{\psi_T f}_{L^p} \norm{\widetilde{\psi_T}\varphi}_{L^{p'}}	\leq \norm{f}_{D^s_{p,\infty}}\sum_{T\in \mathcal{T}}\omega_T^{-s}\norm{\widetilde{\psi_T}\varphi}_{L^{p'}},
\end{align}
with $1/p+1/p'=1$. Applying \eqref{PAPERC:eq:Operatornorm} we get
\begin{align}\label{PAPERC:eq:embeddinglargep2}
\sum_{T\in \mathcal{T}}\omega_T^{-s}\norm{\widetilde{\psi_T}\varphi}_{L^{p'}}&\leq \norm{\left\{\sum_{T'\in \widetilde{T}}\norm{\psi_{T'} \varphi}_{L^{p'}}\right\}_{T\in \mathcal{T}}}_{\ell^1_{\omega^{-s}}}\notag\\
&= \norm{\left\{\left(\norm{\psi_{T} \varphi}_{L^{p'}}\right)^{+}\right\}_{T\in \mathcal{T}}}_{\ell^1_{\omega^{-s}}}\notag\\
&\leq C_+\norm{\left\{\norm{\psi_{T}\varphi}_{L^{p'}}\right\}_{T\in \mathcal{T}}}_{\ell^1_{\omega^{-s}}}=C_+\norm{\varphi}_{D^{-s}_{p',1}}.
\end{align}
Now, \eqref{PAPERC:eq:decomcovtheproof1} implies $\norm{\varphi}_{D^{-s}_{p',1}}\leq C\norm{\varphi}_{D^{\varepsilon-s}_{p',\infty}}$ for $\varepsilon>d$. Hence, since we have already shown that $\mathcal{S}(\bb{R}^d)\hookrightarrow D^s_{p,\infty}$, we conclude from \eqref{PAPERC:eq:embeddinglargep} and \eqref{PAPERC:eq:embeddinglargep2} that $D^s_{p,\infty} \hookrightarrow \mathcal{S}'(\bb{R}^d)$. This proves \Theenuref{PAPERC:The:decomcovthe}{PAPERC:The:decomcovthe1}.

\item \Theenuref{PAPERC:The:decomcovthe}{PAPERC:The:decomcovthe2} follows from \Theenuref{PAPERC:The:decomcovthe}{PAPERC:The:decomcovthe1} and the arguments in \cite[Page 150]{Borup08}.

\item To prove \Theenuref{PAPERC:The:decomcovthe}{PAPERC:The:decomcovthe3} we let $f\in D^s_{p,q}$ and choose a function $I\in C^\infty_c(\bb{R}^d)$ satisfying $0\leq I(x)\leq 1$ and $I(x)\equiv 1$ on some neighbourhood of $x=0$. Since supp($I$) is compact we can choose a \emph{finite} subset $T^*\subset \mathcal{T}$ such that supp$(I)\subset \cup_{T\in T^*}Q_T$ and $\sum_{T\in T^*}\psi_T(x)\equiv 1$ on supp($I$). Hence, with $\widetilde{f}:=I f$ we get
\begin{equation}\label{PAPERC:eq:densityproof1}
\norm{\widetilde{f}}_{L^p}=\norm{\sum_{T\in T^*}\psi_T I f}_{L^p}\leq \sum_{T\in T^*}\norm{\psi_T  f}_{L^p}<\infty,
\end{equation}
since $f\in D^s_{p,q}$. Let $\varphi\in C^\infty_c(\bb{R}^d)$ with $0\leq \varphi(x)\leq 1$ and $\int_{\bb{R}^d}\varphi(x) dx=1$. Also, for $\varepsilon>0$ define $\varphi_\varepsilon(x):=\varepsilon^{-d}\varphi(x/\varepsilon)$ and let $\widetilde{f}_\varepsilon:=\varphi_\varepsilon\ast \widetilde{f}\in \mathcal{S}(\bb{R}^d)$. It follows from \eqref{PAPERC:eq:densityproof1} and a standard result on $L^p$-spaces \cite[Theorem 2.16 on page 64]{Elliot2001} that
\begin{equation*}
\norm{\widetilde{f}-\widetilde{f}_\varepsilon}_{D^s_{p,q}}\leq \norm{\left\{\norm{\widetilde{f}-\widetilde{f}_\varepsilon}_{L^p}\right\}_{T\in \mathcal{T}}}_{\ell^q_{\omega^s}}\rightarrow 0
\end{equation*}
as $\varepsilon\rightarrow 0$. Hence, the proof is done, if we can show that $\|f-\widetilde{f}\|_{D^s_{p,q}}$ can be made arbitrary small by choosing $\tilde{f}$ appropriately. To show this, we define $T_{\circ}:=\{T\in \mathcal{T}~\big|~I(x)\equiv 1 \text{ on }\text{supp}(\psi_T)\}$. Denoting its complement by $T_{\circ}^c$ we get
\begin{equation*}
\norm{f-\widetilde{f}}_{D^s_{p,q}}\leq 2\norm{\left\{\norm{\psi_T f}_{L^p}\right\}_{T\in T_{\circ}^c}}_{\ell^q_{\omega^s}}.
\end{equation*}
Finally, since $f\in D^s_{p,q}$, we can choose supp$(I)$ large enough, such that $\|f-\widetilde{f}\|_{D^s_{p,q}}<\varepsilon$ for any given $\varepsilon>0$. This proves \Theenuref{PAPERC:The:decomcovthe}{PAPERC:The:decomcovthe3}.

\item To prove \Theenuref{PAPERC:The:decomcovthe}{PAPERC:The:decomcovthe4} we first note that $(D^s_{p,q})'\subset \mathcal{S}'(\bb{R}^d)$ since $\mathcal{S}(\bb{R}^d)\subset D^s_{p,q}$. Furthermore, by \Remref{PAPERC:Rem:equivBAPU} we may assume the same BAPU $\{\psi_T\}_{T\in \mathcal{T}}$ is used for both $D^s_{p,q}$ and $D^{-s}_{p',q'}$. Let us first show that $D^{-s}_{p',q'}\subseteq (D^s_{p,q})'$. Given $\sigma\in D^{-s}_{p',q'}$ and $f\in D^s_{p,q}$, applying \eqref{PAPERC:eq:Operatornorm} and H\"{o}lder's inequality twice yield
\begin{align*}
\abs{\scalarp{f,\sigma}}&=\abs{\sum_{T\in \mathcal{T}}\scalarp{\widetilde{\psi_T} f,\psi_T\sigma}}
\leq \sum_{T\in \mathcal{T}}\norm{\widetilde{\psi_T} f}_{L^p}\norm{\psi_T \sigma}_{L^{p'}} \\
&\leq \sum_{T\in \mathcal{T}}\left(\omega_T^s\sum_{T'\in \widetilde{T}}\norm{\psi_{T'} f}_{L^p}\right)\left(\omega_T^{-s}\norm{\psi_T \sigma}_{L^{p'}}\right)\\
&\leq \norm{\left\{\norm{\left(\psi_T  f\right)^+}_{L^{p}}\right\}_{T\in \mathcal{T}}}_{\ell^{q}_{\omega^{s}}}\norm{\left\{\norm{\psi_T \sigma}_{L^{p'}}\right\}_{T\in \mathcal{T}}}_{\ell^{q'}_{\omega^{-s}}}\\
&\leq C_+\norm{f}_{D^s_{p,q}}\norm{\sigma}_{D^{-s}_{p',q'}}.
\end{align*}
To prove that $(D^s_{p,q})'\subseteq D^{-s}_{p',q'}$ we define the space $\ell^q\left(L^p\right)$ as those $\{f_T\}_{T\in \mathcal{T}}\subset \mathcal{S}'(\bb{R}^d)$ satisfying
\begin{equation*}
\norm{\{f_T\}_{T\in \mathcal{T}}}_{\ell^q\left(L^p\right)}:=\norm{\left\{\norm{f_T}_{L^p}\right\}_{T\in \mathcal{T}}}_{\ell^q}<\infty.
\end{equation*}
With this notation we get
\begin{equation*}
\norm{f}_{D^s_{p,q}}=\norm{\left\{\omega_T^s\norm{\psi_T f}_{L^p}\right\}_{T\in \mathcal{T}}}_{\ell^q}=\norm{\{\omega_T^s\psi_T f\}_{T\in \mathcal{T}}}_{\ell^q\left(L^p\right)},
\end{equation*}
for all $f\in D^s_{p,q}$. Since $f\rightarrow \{\omega_T^s\psi_Tf\}_{T\in \mathcal{T}}$ defines an injective mapping from $D^s_{p,q}$ onto a subspace of $\ell^q\left(L^p\right)$, every $\sigma \in (D^s_{p,q})'$ can be interpreted as a functional on that subspace. By the Hahn-Banach theorem, $\sigma$ can be extended to a continuous linear functional on $\ell^q\left(L^p\right)$ where the norm of $\sigma$ is preserved. It thus follows from \cite[Proposition 2.11.1 on page 177]{Triebel2010} that for $f\in D^s_{p,q}$ we may write
\begin{align}
\sigma(f)&=\int_{\bb{R}^d}\sum_{T\in \mathcal{T}}\sigma_T(x)\omega_T^s\psi_T(x)f(x)dx,\quad \text{where}\label{PAPERC:eq:triebelcharacterization}\\
\{\sigma_T(x)\}_{T\in \mathcal{T}}&\in \ell^{q'}\left(L^{p'}\right),\quad \text{and}\quad \norm{\sigma}_*=\norm{\{\sigma_T\}_{T\in \mathcal{T}}}_{\ell^{q'}\left(L^{p'}\right)},\label{PAPERC:eq:triebelcharacterizationdescription}
\end{align}
with $\|\sigma\|_*:=\sup_{\|\{h_T\}\|_{\ell^q\left(L^p\right)}=1}|\sigma(\{h_T\})|$ denoting the standard norm on $\left(\ell^q\left(L^p\right)\right)'$. From \eqref{PAPERC:eq:triebelcharacterization} we conclude that the proof is done if we can show that $\sum_{T\in \mathcal{T}}\sigma_T(x)\omega_T^s\psi_T(x)\in D^{-s}_{p',q'}$. This follows from \eqref{PAPERC:eq:Operatornorm} since
\begin{align*}
\norm{\sum_{T\in \mathcal{T}}\sigma_T\omega_T^s\psi_T}_{D^{-s}_{p',q'}}
&=\norm{\left\{\norm{\psi_T\left(\sum_{T'\in \widetilde{T}}\sigma_{T'}\omega_{T'}^s\psi_{T'}\right)}_{L^{p'}}\right\}_{T\in \mathcal{T}}}_{\ell^{q'}_{\omega^{-s}}}\\
&\leq C\norm{\left\{\norm{\sigma_T\omega_T^s\psi_T}_{L^{p'}}\right\}_{T\in \mathcal{T}}}_{\ell^{q'}_{\omega^{-s}}}\\
&\leq C\norm{\left\{\norm{\sigma_T}_{L^{p'}}\right\}_{T\in \mathcal{T}}}_{\ell^{q'}}\\
&=C\norm{\left\{\sigma_T\right\}_{T\in \mathcal{T}}}_{\ell^{q'}(L^{p'})}=C\norm{\sigma}_*,
\end{align*}
where we use \eqref{PAPERC:eq:triebelcharacterizationdescription} in the last equation. This proves \Theenuref{PAPERC:The:decomcovthe}{PAPERC:The:decomcovthe4}.
\end{enumerate}
\end{proof}

\bibliographystyle{abbrv}

\end{document}